\documentclass[a4paper,12pt]{amsart}
\input{preamble.tex}

\title{AKE principles for deeply ramified fields}
\author{Franziska Jahnke and Jonas van der Schaaf}
\thanks{The research leading to this work was 
supported by the Deutsche Forschungsgemeinschaft (DFG, German Research Foundation)
via the ANR-DFG grant AKE Pact (project number 545528554)
and under Germany's Excellence Strategy EXC 2044/2 -390685587, Mathematics M\"unster: Dynamics-Geometry-Structure. The results of the article are part of the second author's MSc thesis
supervised by the first author.}

\begin{document}

\begin{abstract}
    We study the model theory of deeply ramified fields of positive 
    characteristic. Generalizing the 
    perfect case treated in work by Jahnke and Kartas on the model theory
    of perfectoid fields, we obtain Ax--Kochen/Ershov 
    principles for certain deeply ramified fields of positive characteristic and fixed degree of imperfection. Our results apply in particular to all
    deeply ramified henselian valued fields of rank $1$.
\end{abstract}
	\maketitle
    

\section{Introduction}
A fundamental aim in mathematics is to understand complicated objects via
simpler invariants. In the model theory of valued fields, this usually takes the
form of an Ax--Kochen/Ershov principle, after Ax and Kochen \cite{AxKo} and,
independently, Ershov \cite{Er65} who proved that the theory of a henselian
valued field of residue characteristic $0$ (in the language
$\mathcal{L}_\mathrm{val}$ of valued fields)\footnote{These notations will be
    explained below.} is entirely determined by the theories of its
residue field (in the language $\mathcal{L}_\mathrm{ring}$ of rings) and the
theory of its value group (in $\mathcal{L}_\mathrm{oag}$ of ordered abelian
groups). Analogous results hold for unramified henselian valued fields
(\cite{AxKo}, \cite{AJ19}), algebraically maximal Kaplansky valued fields
(\cite{ZiePhD} and \cite{Er66}, or \cite{ErBook} for an English version of the latter), 
tame valued fields
of positive characteristic (\cite{Kuh16}), and separably tame valued fields of a
fixed degree of imperfection (\cite{KP}, \cite{A25}).

The typical way to prove an AKE principle is via an embedding lemma, which
allows one to prolong given embeddings of value group and residue field to an
embedding of valued fields (cf.~\cite[\S 4.6]{PD}, \cite{Dri14}, \cite{Kuh16}). 
An embedding lemma may also
yield quantifier elimination down to residue field and value group, cf.~\cite{Pas89} - but note that no relative quantifier result is known for tame fields
of positive characteristic. Moreover, an embedding lemma usually
implies a variety of AKE principles, most prominently:
\[
    \textbf{AKE}_\prec : (K,v)\prec (L,w) \Longleftrightarrow Kv\prec Lw \wedge vK\prec wL,
\]
and
\[
    \textbf{AKE}_{\prec_\exists} : (K,v)\prec (L,w) \Longleftrightarrow Kv\prec_{\exists} Lw \wedge vK\prec_{\exists} wL,
\]
Both of these principles hold in all the examples discussed
above. Sometimes an embedding lemma only holds over a common subfield with
additional side conditions on the embeddings. This happens in particular in
the case of tame fields of mixed characteristic, where then only the
$\textbf{AKE}_\prec$ and $\textbf{AKE}_{\prec_\exists}$ principles hold but
the regular AKE principle $\textbf{AKE}_{\equiv}$ fails, even under natural further assumptions (see both \cite{ketelsenCompositionAxKochenErshov2026} and its appendix by Dittmann).

When a valued field $(K,v)$ admits a proper immediate algebraic extension $(L,w)$,
i.e., a proper algebraic extension such that the induced embeddings of 
value group and residue field are the identity, both $\textbf{AKE}_\prec$ and
$\textbf{AKE}_{\prec_\exists}$ must fail for this pair. Nonetheless, in the case
of perfectoid fields with pseudouniformizer $\varpi$ (and more generally for
certain elementary classes of perfect deeply ramified henselian valued fields
$(K,v)$ with an distuingished element $\varpi \in \mathfrak{M}_v \setminus
    \{0\}$), Jahnke and Kartas \cite{JK23} prove AKE principles down to the pointed
value group $(vK, v\varpi)$ and residue ring $\mathcal{O}_v/(\varpi)$. In mixed
characteristic, these only work over a base field (essentially due to the
failure of $\textbf{AKE}_\equiv$ for tame fields of mixed characteristic). In
positive characteristic, Jahnke and Kartas prove versions of all three AKE
principles.

The aim of this work is now to remove the perfectness assumption in the AKE
principles shown in \cite{JK23}, which of course means we are necessarily
working in positive characteristic. More precisely, we show first that the
relevant class of valued fields is elementary:

\begin{thm*}[see \Cref{cor:elem}]
    Fix a prime \(p\) and \(e\in\N \cup\{\infty\}\). The class \(\C_{p,e}\) of pointed valued
    fields \((K,v,t)\) satisfying the following properties is elementary:
    \begin{enumerate}[(i)]
        \item \(K\) has characteristic \(p\) and degree of imperfection
              \(e\),
        \item \((K,v)\) is deeply ramified,
        \item \(0<vt<\infty\) and the valuation on $K$ with valuation ring
        \(\O_{v}[1/t]\) is separably tame.
    \end{enumerate}
\end{thm*}

This class contains in particular all deeply ramified fields of rank $1$, where
$t$ may be chosen as any element in $\mathfrak{M}_v\setminus \{0\}$. We also
give an example of a deeply ramified field that is not contained in this class.
We then use a similar proof technique to \cite{JK23} to show the following main
theorem:

\begin{thm*}[see \Cref{thm:ake}]
    Fix a prime \(p\) and \(e\in\N \cup \{\infty\}\). Let \((K,v,t)\subseteq(L,w,t)\) 
    and \((K,v,t)\subseteq (L',w',t')\)
    be separable extensions of pointed valued fields in the aforementioned class.
    Furthermore, assume that
    \begin{enumerate}[(i)]
        \item \(vK\preceq_{\exists}wL\); and
        \item \(L,L'\) have the same relative degree of imperfection over \(K\):
              \([L:K^{p}]=[L':K^{p}]\)
    \end{enumerate}
    hold. Then the following
    are equivalent
    \begin{enumerate}
        \item \((L,w,t)\equiv_{(K,v,t)}(L',w',t')\);
        \item \(\O_{w}/(t)\equiv_{\O_{v}/(t)}\O_{w'}/(t')\) and
              \(wL\equiv_{vK}w'L'\).
    \end{enumerate}
\end{thm*}

This leads to appropriate versions of $\textbf{AKE}_\equiv$,
$\textbf{AKE}_\prec$ and $\textbf{AKE}_{\prec_\exists}$ for each class
$\mathcal{C}_{e,p}$.

As an application, we study a variant of the major open problem whether the
embedding of $\mathbb{F}_p(t)^h = \mathbb{F}_p((t))\cap
    \mathbb{F}_p(t)^\mathrm{alg}$ into $\mathbb{F}_p((t))$ is elementary. We show
that whenever $K$ is a deeply ramified algebraic extension of
$\mathbb{F}_p(t)^h$, then the inclusion is an elementary embedding
$$K \prec K . \mathbb{F}_p((t)).$$ In particular, if we denote by
$K^\mathrm{AS}$ the Artin-Schreier hull of a field $K$, i.e., the Galois extension of 
$K$ obtained by (repeatedly) adding all roots of Artin-Schreier
polynomials to arrive at a field that admits no Artin-Schreier extensions, then
we get \Cref{cor:as-closure-nice}:
$$(\mathbb{F}_p(t)^h)^\mathrm{AS} \prec \mathbb{F}_p((t))^\mathrm{AS}.$$

\subsection*{Notation}

    In this paper, we write \(\mathcal{L}_{\text{ring}}\) for the language
    of rings \((0,1,+,-,\cdot)\), we write \(\mathcal{L}_{\text{oag}}\) for the
    language \((0,+,-,<)\) of ordered abelian groups. Finally, the language of
    valued fields \(\mathcal{L}_{\text{val}}\) is the language of rings with a
    predicate \(\O\) for the elements of the valuation ring. For a valued field $(K,v)$ and an element $t\in K$ we write $(K,v,t)$ for the $\mathcal{L}_\text{val}(t)$-structure where
    we name $t$ as a constant, and call $(K,v,t)$ a pointed valued field.

    We write \(\O_{v}\) for the valuation ring of a given valuation $v$ and
    \(\mathfrak{M}_{v}\) for its maximal ideal.

\begin{defi}
    Let \((K,v)\) be a valued field and \(t\in\mathfrak{M}_{v}\). We write
    \(v_{t^{+}}\) for the finest coarsening of \(v\) such that
    \(v_{t^{+}}(t)=0\) and \(v_{t^{-}}\) for the coarsest coarsening such that
    \(v_{t^{-}}(t)>0\). These give rise to a diagram of places where we have
    marked the associated value groups
    \[
        K \xrightarrow{vK/\Delta^{+}} Kv_{t^{+}} \xrightarrow{\Delta^{+}/\Delta^{-}} Kv_{t^{-}} \xrightarrow{\Delta^{-}} Kv.
    \]
    Here \(\Delta^{+}\) is the smallest convex subgroup of \(vK\) such that
    \(vt\in\Delta^{+}\) and \(\Delta^{-}\) the largest convex subgroup such that
    \(vt\notin\Delta^{-}\). In particular the group \(\Delta^{+}/\Delta^{-}\) is
    a totally ordered group of rank \(1\).

    Some objects associated to this are particularly easy to describe. The
    valuation ring \(\O_{v_{t^{+}}}\) can be identified with \(\O_{v}[1/t]\):
    the latter is the smallest overring (which is automatically a valuation
    ring) of \(\O_{v}\) where \(t\) is invertible. Similarly, we can describe
    \(Kv_{-}\) as the fraction field of \((\O_{v}/(t))_{\textrm{red}}\): the
    maximal ideal of \(\O_{v_{t^{-}}}\) is the smallest prime ideal containing
    \(t\). This is exactly the radical of \((t)\).
\end{defi}

For fields $K$ and $L$
contained in a common superfield $M$, we write $K.L$ for their compositum.

\subsection*{Acknowledgements}
We wish to thank Sylvy Anscombe, Philip Dittmann, and Konstantinos Kartas for their many helpful 
comments and suggestions. 

\section{Deeply ramified and almost separably tame fields}

At the core of the main theorem of this paper are two properties of valued
fields: being deeply ramified and almost separably tame. In this section, we will
exhibit an elementary class of fields with these properties and argue that this
is a reasonable class to work in.

\subsection{Separably tame valued fields}

The theory of (separably) tame valued fields is a key recent development in
the model theory of valued fields. Pioneered by Kuhlmann \cite{Kuh16} and
Kuhlmann and Pal \cite{KP}, the class of (separably) tame valued fields is one
where a wide variety of AKE principles hold for (separably) tame valued fields
of finite degree of imperfection. Recently, Anscombe expanded their work to
include the case of infinite degree of imperfection \cite{A25}. We give a brief
explanation of the model-theoretic results for separably tame valued fields as
we will need them later on.

\begin{defi}
    Let \((K,v)\) be a henselian valued field of residue characteristic \(p>0\)
    and \((L,w)/(K,v)\) a finite field extension. This extension is called tame
    if the following conditions hold:
    \begin{enumerate}[(i)]
        \item \(p\nmid(wL:vK)\);
        \item \(Lw/Kv\) is separable;
        \item the extension is defectless.
    \end{enumerate}
    A valued field is then called separably tame if all separable finite field
    extensions are tame. Trivially valued fields are always separably tame. For
    non-trivially valued fields, the property of being separably tame is
    equivalent to the following properties:
    \begin{enumerate}[(i)]
        \item \(vK\) is \(p\)-divisible;
        \item \(Kv\) is perfect;
        \item \((K,v)\) is separable-algebraically maximal.
    \end{enumerate}
\end{defi}

The main result in the theory of separably tame valued fields are their AKE
principles \cite[Theorem 1.2]{KP}. The infinite degree of
imperfection case was proven in \cite[Theorem 4.22(II) and Theorem 4.24]{A25}.

\begin{thm}[Kuhlmann for $e=0$, Kuhlmann--Pal for $e<\infty$, Anscombe for $e=\infty$]
    Let \((L,w),(L',w')\) be two separably tame valued fields of positive
    characteristic such that \(L,L'\) have the same degree of imperfection $e
        \in \mathbb{N}\cup \{\infty\}$. Then the following are equivalent
    \begin{enumerate}[(i)]
        \item \((L,w)\equiv(L',w')\);
        \item \(wL\equiv w'L'\) and \(Lw\equiv L'w'\).
    \end{enumerate}
    Moreover, assume that \((L,w)\) and \((L',w')\) share a common separably
    tame subfield \((K,v)\) with \(L/K\) and \(L'/K\) both separable extensions
    such that $wL/vK$ is torsion-free and, if $e = \infty$, assume further that
    the relative imperfection degrees of $L/K$ coincides with that of $L'/K$.
    Then we get equivalences between the following statements:
    \begin{enumerate}[(i)]
        \item \((L,w)\equiv_{(K,v)}(L',w')\);
        \item \(wL\equiv_{vK}w'L'\) and \(Lw\equiv_{Kv}L'w'\).
    \end{enumerate}
\end{thm}

\subsection{Deeply ramified valued fields}

Deeply ramified valued fields have been extensively studied in the context of
almost mathematics \cite{GR03} and more recently, the model theory of valued fields \cite{KR21, BOISSONNEAU_2024,
    JK23, CK25, cutkoskyComputationKahlerDifferentials2025}. Similarly, we will use them as the foundation of the main theorem.

\begin{defi}
    \label{def:deeply-ramified}
    Let \((K,v)\) be a valued field of characteristic \(p\). It is called deeply
    ramified if one of the following equivalent conditions hold:
    \begin{enumerate}[(i)]
        \item \(\Omega_{\O_{v^{\text{sep}}}/\O_{v}}=0\);
        \item the completion \(\hat{K}\) is perfect;
        \item \(K\) is dense in its perfect hull;
        \item for all \(x\in K\) with \(0<vx\), the ring \(\O_{v}/(x)\) is
              semi-perfect.
    \end{enumerate}

    \begin{proof}
        The equivalence of the first three statements is standard and can be
        found in \cite{KR21}.

        We prove (iii) \(\Leftrightarrow\) (iv).

        For the left to right implication, there is an isomorphism
        \(\O_{v}/(x)\cong\O_{\hat{v}}/(x)\). Therefore if the Frobenius map is
        surjective on \(\hat{K}\) (and hence \(\O_{\hat{v}}\)), then it is also
        surjective on \(\O_{v}/(x)\).

        For the right to left implication, consider a basic closed ball around
        \(\sqrt[p^{r}]{a}\) for \(a\in K\) and \(r\in\N\) of radius
        \(\varepsilon\):
        \[
            B=\{b\in K|v(x-\sqrt[p^{r}]{a})\geq\varepsilon\}.
        \]
        Take \(x\in K\) with \(v(x)\geq p^{r}\cdot\varepsilon\). Then
        \(\O_{v}/(x)\) is semi-perfect so there is a \(b\in K\) with
        \(v(b^{p}-a)\geq v(x)\). A simple verification shows that \(b\in B\):
        \begin{align*}
            v(b-\sqrt[p^{r}]{a}) & =\frac{1}{p^{r}}v\left(b^{p^{r}}-a\right) \\
                                 & \geq\frac{p^{r}\cdot\varepsilon}{p^{r}}   \\
                                 & =\varepsilon.
        \end{align*}
    \end{proof}
\end{defi}

\subsection{Unramified extensions}

In the literature, there are multiple different definitions of unramified
extensions. Here, unramified extensions are defectless and must induce a
separable extension of the residue field.

\begin{defi}
    A finite extension \((K',v')/(K,v)\) of henselian valued fields is
    unramified if \(f(v'/v)=[K':K]\) and \(K'v'/Kv\) is separable.
\end{defi}

There are a variety of equivalent definitions of unramified extensions. The ones
we will use are expressed using the different: a quantity associated to an
element of an algebraic field extension similar to the discriminant.

\begin{defi}
    Let \(K\) be a field and \(a\in\overline{K}\) an element of the algebraic
    closure. Write \(m_{a}(x)\) for its minimal polynomial over $K$. We write
    \(\delta(a)=m_{a}'(a)\) for the different of \(a\). Factoring \(m_{a}(x)\)
    as
    \[
        m_{a}(x)=\prod_{i}(x-a_{i})
    \]
    in $\overline{K}$ with \(a_{0}=a\), using the product rule we can also write the different as
    \[
        \delta(a)=\prod_{i>0}(a_{i}-a).
    \]
\end{defi}

\begin{lem}
    \label{lem:unramif}
    Let \((K,v)\) be a henselian valued field and \((K',v')/(K,v)\) a finite
    extension. Then the following are equivalent:
    \begin{enumerate}[(i)]
        \item The extension \((K',v')/(K,v)\) is unramified;
        \item there exists an \(a\in\O_{v'}\) such that \(\O_{v'}=\O_{v}[a]\)
              and \(\delta(a)\in\O_{v}^{\times}\);
        \item there exists an \(a\in\O_{v'}\) such that \(K'=K(a)\) and
              \(\delta(a)\in\O_{v}^{\times}\);
        \item there exists an \(f\in K[X]\) such that
              \(\discr(f)\in\O_{v}^{\times}\) and \(K'=K[X]/(f)\);
        \item \(\O_{v'}/\O_{v}\) is étale.
    \end{enumerate}

    \begin{proof}
        See \cite[Fact 2.4.7]{JK23}.
    \end{proof}
\end{lem}

\subsection{Almost separably tame valued fields}

The concept of almost separably tame valuations is one of the main tools we
use to demonstrate the AKE principle.
These allow us to split up valuations into different
relatively nice parts that we can understand separately. Under certain
conditions, we can use model theory to characterize such fields.

\begin{defi}
    \label{def:sep-tameable}
    Let \((K,v)\) be a valued field. It is called almost separably tame if it has
    an \(\aleph_{1}\)-saturated elementary extension
    \((K^{*},v^{*})\succeq(K,v)\) which has a non-trivial coarsening \(v^{*+}\)
    such that \((K,v^{*+})\) is separably tame.
\end{defi}

\begin{lem}
    \label{lem:diff-cont}
    Let \((K,v)\) be a valued field and \(a\in K^{\text{sep}}\) an element of
    its separable closure with minimal polynomial \(m_{a}(x)\). Then the
    different \(\delta(a)\) is continuous in the coefficients of \(m\).

    \begin{proof}
        Let \(a=a_{0},\ldots,a_{n}\) be the roots of \(m_{a}\). Then
        \[
            \delta(a)=m_{a}'(a)=\prod_{i>0}(a_{0}-a_{i})
        \]
        is a polynomial in the roots of \(m_{a}\). By
        \cite[Theorem 2.4.7]{EP05}, the roots of \(m_{a}\) are
        continuous in the coefficients of \(m_{a}\) and so \(\delta(a)\) is as
        well.
    \end{proof}
\end{lem}

\begin{lem}
    \label{lem:small-diff-approx}
    Let \((K,v,t)\) be a pointed deeply ramified henselian valued field of
    characteristic \(p>0\) such that \(\{v(t^{n})\;|\;n\in\N\}\) is
    cofinal in \(vK\). Then for any finite separable field extension \(L/K\)
    there is an \(\alpha\in L\) such that \(L=K(\alpha)\) and \(0\leq
    v(\delta(a))\leq v(t)\).

    \begin{proof}
        Let \(L/K\) be such an extension and take any generator \(\alpha\in
        \O_{v}\) which exists by separability. This will satisfy
        \(v(\delta\alpha)\geq0\).

        We get a diagram of valued fields
        \[
            \begin{tikzcd}
                L & L.K^{\frac{1}{p^{\infty}}} \\
                K & K^{\frac{1}{p^{\infty}}}
                \arrow[from=2-1,to=1-1]
                \arrow[from=2-1,to=2-2]
                \arrow[from=2-2,to=1-2]
                \arrow[from=1-1,to=1-2]
            \end{tikzcd}
        \]
        The extension \(K^{\frac{1}{p^{\infty}}}/K\) is immediate and purely
        inseparable, therefore it is linearly disjoint from \(L\) over \(K\).
        Therefore, we get that
        \([L.K^{\frac{1}{p^{\infty}}}:K^{\frac{1}{p^{\infty}}}]=[L:K]\). In
        particular \(f_{K}^{\alpha}\) is still irreducible over
        \(K^{\frac{1}{p^{\infty}}}\) and
        \(L.K^{\frac{1}{p^{\infty}}}=K^{\frac{1}{p^{\infty}}}(\alpha)\). By
        separability \(\delta \alpha=(f_{K}^{\alpha})'(\alpha)\neq0\) so
        \(v(\delta \alpha)<v(t^{n})\) for some \(n\in\N\).

        We claim that there is an \(\alpha'\in K^{\frac{1}{p^{\infty}}}\) such
        that \(L.K^{\frac{1}{p^{\infty}}}=K^{\frac{1}{p^{\infty}}}(\alpha')\)
        and \(\delta \alpha'=\frac{1}{p^{N}}\delta \alpha\leq v(t)\). Let
        \(N=\lceil\log_{p}(n)\rceil\). Then the \(p^{N}\)th root
        \(\sqrt[p^{N}]{\alpha}\) does the trick as
        \[
            \delta(\sqrt[p^{N}]{\alpha})=m_{\sqrt[p^{N}]{\alpha}}'(\sqrt[p^{N}]{\alpha})=\sqrt[p^N]{m_{\alpha}'(\alpha)}=\sqrt[p^N]{\delta(\alpha)}.
        \]

        Using \Cref{lem:diff-cont},
        and Krasner's lemma \cite[Theorem 4.1.7]{EP05}, we find that a
        sufficiently good monic approximation \(f\) of
        \(m_{\sqrt[p^{N}]{a}}(x)\) in \(K[X]\) has a root \(\beta\) in \(L\)
        such that we have \(K(\beta)=K(\alpha)\) and
        \[
            0\leq v(\delta(\beta))=v(\delta(\sqrt[p^{N}]{\alpha}))=\frac{1}{p^{N}}v(\delta(\alpha))\leq v(t).
        \]
    \end{proof}
\end{lem}

Analogously to \cite[Proposition 4.1.4]{JK23}, we can now prove that in deeply
ramified henselian valued fields, being almost separably tame with a parameter
\(t\in K\) is elementary by showing it is equivalent to a property which is
elementary.

\begin{prop}
    \label{prop:elem-prop}
    Let \((K,v,t)\) be a pointed deeply ramified henselian valued field of
    characteristic \(p\) with \(0<vt<\infty\). Then the following are
    equivalent:
    \begin{enumerate}[(i)]
        \item The valued field \((K,v_{t^{+}}))\) is separably defectless;
        \item For every finite separable field extension \((K',v')/(K,v)\) with
              \(e(v'/v)=1\), there is an \(\alpha\in L\) such that
              \(L=K(\alpha)\) and \(0\leq v(\delta(\alpha))\leq vt\);
        \item The valued field \((K,v_{t^{+}}))\) is separable-algebraically
              maximal.
    \end{enumerate}

    \begin{proof}
        We distinguish between the case whether \(\O_{v_{t^{+}}}\) is the
        trivial valuation ring.

        If it is the trivial valuation ring, then all three statements are
        simply true. The first and last ones are trivially true, and (ii) is
        true by \Cref{lem:small-diff-approx}.

        Now because \((K,v)\) is deeply ramified and \(\O_{v_{t^{+}}}\) is
        non-trivial the residue field \(Kv_{t^{+}}\) is perfect. We will use
        this for the rest of the proof.

        (i) \(\implies\) (ii): Let \(K'/K\) be an extension with
        \(e(v'/v)=1\). Then because of the equality\(e(v'/v)=e(v'_{t^{+}}/v_{t^{+}})\cdot e(\overline{v'}/\overline{v})\), we know that
        \(e(v'_{t^{+}}/v_{t^{+}})=1\) as well. This means that the extension is
        unramified with respect to \(v_{t^{+}}\) as \((K,v_{t^{+}})\) is
        separably defectless. We can then apply \Cref{lem:small-diff-approx}
        again\footnote{Technically speaking we can use a much weaker version of
            the lemma as \(Kv_{t^{+}}\) is already perfect and so we can replace
            $\overline{\alpha}$ by an appropriate $p^n$th root, see \cite[Lemma 4.1.2]{JK23}.} to \((Kv_{t^{+}},\overline{v})\) to get a generator
        \(\overline{\alpha}\) of the field extension \(K'v'_{t^{+}}/Kv_{t^{+}}\)
        with \(0\leq\overline{v}\delta(\overline{\alpha})\leq v\overline{t}\).
        Lifting \(\overline{\alpha}\) from the residue field to \(K'\) gives a
        desired generator.

        (ii) \(\implies\) (iii) Let \(K'/K\) be immediate with respect to
        \(v_{t^{+}}\). Then by the Lemma of Ostrowski
        \[
            [K':K]=p^{\nu}e(v'_{t^{+}}/v_{t^{+}})f(v'_{t^{+}}/v_{t^{+}})=p^{\nu}.
        \]
        This implies \(e(v'/v)\mid[K':K]=p^{\nu}\), but \(vK\) is already
        \(p\)-divisible hence \(e(v'/v)=1\). We get an \(\alpha\) as in (ii) and so by \Cref{lem:unramif} the extension is unramified and therefore
        \([K':K]=f(v'/v)=1\).

        (iii) \(\implies\) (i) The valuation ring \((K,v_{t^{+}})\) has perfect
        residue field, is separable-algebraically maximal and \(v_{t^{+}}K\) is
        \(p\)-divisible. This means that \((K,v_{t^{+}})\) is separably tame and
        a fortiori separably defectless.
    \end{proof}
\end{prop}

\begin{kor}
    \label{cor:elem}
    The class of pointed deeply ramified henselian valued fields \((K,v,t)\) of
    characteristic \(p\) such that \(\O_{v_{t^{+}}}\) is separably tame is
    elementary. We will write this class as \(\C_{p}\).

    We can divide this class into subclasses depending on the degree of
    imperfection, where we will write \(\C_{p,e}\) for fields in \(\C_{p}\) with
    (possibly infinite) degree of imperfection \(e\).

    All of these classes are elementary by (ii) of \Cref{prop:elem-prop}.
\end{kor}

By definition, a valued field \((K,v,t)\) in \(\C_{p}\) has
the property that \((K,v_{t^{+}})\) is separably tame. If \((K,v,t)\) is
\(\aleph_{1}\)-saturated, then \((K,v_{t^{-}})\) is also separably tame.

\begin{lem}
    \label{lem:sep-tame-coarsening}
    Take a pointed valued field \((K,v,t)\in\C_{p,e}\) which is
    \(\aleph_{1}\)-saturated. Then \((K,v_{t^{-}})\) is separably tame.

    \begin{proof}
        We prove that the valued field \((Kv_{t^{+}},\overline{v_{t^{-}}})\) is
        tame. Then \(v_{t^{-}}\) is the composition of a separably tame and a
        tame valuation and therefore separably tame itself.

        By \cite[Theorem 1.13]{AK16}, the valued field \((Kv_{t^{+}},\overline{v_{t^{-}}})\)
        is spherically defectless and must have value group \(\mathbb{Z}\) or \(\mathbb{R}\). Because the value group of \((K,v)\) is \(p\)-divisible, we conclude that \((Kv_{t^{+}},\overline{v_{t^{-}}})\)
        must be \(\mathbb{R}\).

        This means that \((Kv_{t^{+}},\overline{v_{t^{-}}})\) has
        \(p\)-divisible value group, perfect residue field and is defectless.
    \end{proof}
\end{lem}

If we restrict to \(e=0\), we obtain the characteristic \(p\) case of
\cite{JK23}. As pointed out before, there pointed valued fields are studied
which are also deeply ramified and henselian and almost tame, but the involved
fields are all perfect.

We now give an example of a valued field which is not separably tame, but is
contained in \(\C\).

\begin{examp}[Pointed out to us by Philip Dittmann]
    \label{ex:elem-of-class}
    Let \(p\) be an odd prime. The pointed valued field
    \((\asclose{\F_{p}\powerfield{t}},v,t)\) is in \(\C\). It is deeply ramified
    and rank \(1\), so the coarsening \(\O_{v_{t^{+}}}\) is the trivial
    valuation ring which is defectless. The field itself however is not
    separably tame.

    We claim that separable extension
    \(\mathbb{F}_{p}\powerfield{t}(\sqrt{t})\left(\rho_{\frac{1}{\sqrt{t}}}\right)\) is
    linearly disjoint from \(\asclose{\F_p\powerfield{t}}\). Then the composite extension
    \(\F_{p}\powerfield{t}(\sqrt{t})\left(\rho_{\frac{1}{\sqrt{t}}}\right).\asclose{\F_{p}\powerfield{t}}/\asclose{\F_{p}\powerfield{t}}\)
    must have defect: both the value group and residue field of
    \(\asclose{\F_{p}\powerfield{t}}\) have no degree \(p\) extension. The value
    group is \(\frac{1}{p^{\infty}}\Z\) which is \(p\)-divisible, and the
    residue field \(\asclose{\F_{p}}\) has no degree \(p\)-extensions either.

    Now to show linear independence we use some basic Galois theory. In this
    case, it is sufficient to show that
    \(\F_{p}\powerfield{t}(\sqrt{t})\left(\rho_{\frac{1}{\sqrt{t}}}\right)\cap\asclose{\F_{p}\powerfield{t}}\)
    is \(\F_{p}\powerfield{t}\). Therefore, we assume that the intersection is a
    non-trivial intermediate extension as in the diagram below.
    
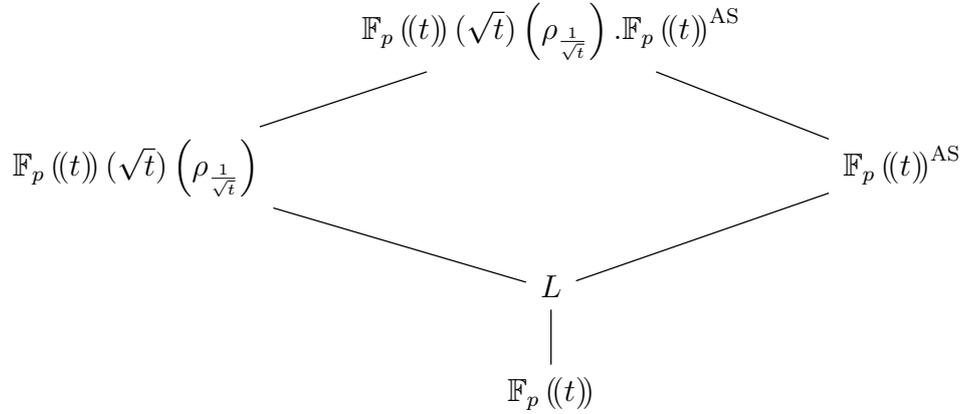
\[
        \begin{tikzcd}[]
            &\F_{p}\powerfield{t}(\sqrt{t})\left(\rho_{\frac{1}{\sqrt{t}}}\right).\asclose{\F_{p}\powerfield{t}} & \\
            \F_{p}\powerfield{t}(\sqrt{t})\left(\rho_{\frac{1}{\sqrt{t}}}\right) & & \asclose{\F_{p}\powerfield{t}} \\
            & L & \\
            & \F_{p}\powerfield{t}
            \arrow[from=4-2,to=3-2, no head]
            \arrow[from=3-2,to=2-1, no head]
            \arrow[from=3-2,to=2-3, no head]
            \arrow[from=2-1,to=1-2, no head]
            \arrow[from=2-3,to=1-2, no head]
        \end{tikzcd}
    \]
    The field extension \(L/\F_{p}\powerfield{t}\) must have degree \(p\):
    \(\asclose{\F_{p}\powerfield{t}}/\F_{p}\powerfield{t}\) is a tower of degree
    \(p\) extensions and \(\F_{p}\powerfield{t}(\sqrt{t})(\rho_{\frac{1}{\sqrt{t}}})/\F_{p}\powerfield{t}\)
    has degree \(2p\). This means that \(L/\F_{p}\powerfield{t}\) has degree
    \(p\). It is also a Galois extension because it is the intersection of two
    Galois extensions.

    The Galois group of this extension is non-commutative. Consider
    generators of the Galois group which act in the following manner:
    \begin{align*}
        \sigma:\sqrt{t}&\mapsto-\sqrt{t}, \\
        \tau:\rho_{\frac{1}{\sqrt{t}}}&\mapsto\rho_{\frac{1}{\sqrt{t}}}+1.
    \end{align*}
    We know that \(\sigma(\rho_{\frac{1}{\sqrt{t}}})=-\rho_{\frac{1}{\sqrt{t}}}\)
    because \(\sigma(\rho_{\frac{1}{\sqrt{t}}})\) must be an Artin-Schreier root of
    \(\sigma(\sqrt{t})=-\sqrt{t}\) and the sum of Artin-Schreier roots of \(a,b\)
    is an Artin-Schreier root of $a+b$. Therefore
    \(\sigma(\rho_{\frac{1}{\sqrt{t}}})+\rho_{\frac{1}{\sqrt{t}}}\) is an Artin-Schreier
    root of $0$ and hence
    \(\sigma(\rho_{\frac{1}{\sqrt{t}}})=-\rho_{\frac{1}{\sqrt{t}}}+n\) for some \(n\in\F_{p}\). A simple computation then gives
    \begin{align*}
        \sigma(\tau(\rho_{\frac{1}{\sqrt{t}}}))&=\sigma(\rho_{\frac{1}{\sqrt{t}}}+1)\\
                                               &=-\rho_{\frac{1}{\sqrt{t}}}+n+1 \\
        \tau(\sigma(\rho_{\frac{1}{\sqrt{t}}}))&=\tau(-\rho_{\frac{1}{\sqrt{t}}})\\
                                               &=-\rho_{\frac{1}{\sqrt{t}}}-(n+1).
    \end{align*}

    The Galois group
    \(\Gal\left(\F_{p}\powerfield{t}(\sqrt{t})(\rho_{\frac{1}{\sqrt{t}}})/\F_{p}\powerfield{t}\right)\)
    must then be the dihedral group \(D_{p}\) of order \(2p\), as it is the only
    non-commutative group of order \(2p\).
    This group has no normal subgroup of
    order \(2\), so by the fundamental theorem of Galois theory the non-trivial
    intermediate field \(L\) cannot exist.

    This demonstrates that \(\asclose{\F_{p}\powerfield{t}}\) with the
    \(t\)-adic valuation is not separably defectless.
\end{examp}

\begin{examp}
    There also exist fields which are not almost separably tame. Let \(\U\) be a
    non-principal ultrafilter on \(\N\) and \(p\) an odd prime. We claim that
    \(K=\asclose{\F_{p}\powerfield{t}^{\U}}\) is such a field with a valuation
    which we call \(v\), induced by the ultraproduct of the valuation
    on \(\F_{p}\powerfield{t}\).

    It suffices to show that for any \(x\in \mathfrak{M}_{v}\), the valuation
    \(v_{x}^{+}\) is not separably tame. To see this, suppose there is some
    non-trivial convex subgroup \(\Delta\subseteq vK\), then there is a
    \(\gamma\in vK\setminus\Delta\). Take \(x\) with \(vx=\gamma\). If \(v_{x}^{+}\) is
    not separably-tame, then neither is the coarsening corresponding to
    \(\Delta\) because the former is a coarsening of the latter. We can even
    reduce to the case that \(x\in(\F_{p}\powerfield{t})^{\U}\): the Archimedean
    components of the value groups of both \(\F_{p}\powerfield{t}^{\U}\) and its
    Artin-Schreier hull are the same.

    Now, let \((x_{i})_{i}\in\F_{p}\powerfield{t}^{\U}\) be as above. We choose a sequence
    \(y=(y_{i})_{i}\in\F_{p}\powerfield{t}^{\U}\) such that for all \(i\):
    \(vy_{i}>i\cdot v(x)\) and \(2\nmid vy_{i}\). Then the field extension
    \(K(\sqrt{y})\left(\rho_{\frac{1}{\sqrt{y}}}\right)\) gives us an example of
    separable defect.

    Because \(y\in\mathfrak{M}{v_{x}^{+}}\), by the same reasoning as before,
    the degree \(p\)-part of the extension must have separable defect: both the
    residue field \(\asclose{\F_{p}}\) and value group
    \(\frac{1}{p^{\infty}}\Z^{\U}\) have no degree \(p\) extensions.
\end{examp}

Something even stronger is true for the above valued field: no elementary
extension is almost separably tame.

\begin{examp}
    No elementary extension of \(K=\asclose{\F_{p}\powerfield{t}^{\U}}\) is
    almost separably tame.

    We may assume that the elementary extensions are ultrapowers: by the
    Keisler-Shelah theorem, for any elementarily equivalent valued field
    \((K',v')\) we can find an ultrafilter \(\U'\) on some set such that
    \((K,v)^{\U'}\cong (K',v')^{\U'}\). Then, if \((K',v')\) did contain a parameter
    \(t'\) such that \((K',v',t')\in\C_{p}\), by elementarity, the ultraproduct
    would be in \(\C_{p}\) as well.

    Take \((x_{i})_{i}\in\mathfrak{M}_{v'}\). Then for all
    \(x_{i}\in\mathfrak{M}_{v}\), by the previous example we can find a
    separable extension of degree \(p\) which has defect with respect to the
    coarsening \(v_{x_{i}}^{+}\). For a fixed (or bounded) degree, this is an
    elementary property and by Los' theorem, the element \((x_{i})_{i}\) must
    have the same property.
\end{examp}

The class \(\C_{p}\) 
is surprisingly large since being deeply ramified passes down 
from (non-trivial) coarsenings in positive characteristic, so
 many
 fields are automatically deeply ramified.

\begin{prop}
    \label{prop:tameable-ramif}
    Let \((K,v)\) be a non-trivially valued field of characteristic \(p>0\) with
    a non-trivial coarsening \(v^{+}\) which is separably tame. Then \(K\) is
    deeply ramified.

    \begin{proof}
        As separably tame valued fields are deeply ramified \cite[Theorem
            1.2]{KR21}, so is \(\O_{v^{+}}\). By \cite[Theorem 2.3.4]{EP05},
        both \(\O_{v}\) and its coarsening induce the same topology on \(K\)
        and hence \((K,v)\) is also dense in its perfect hull.
    \end{proof}
\end{prop}

It is important to note that it is very important that \(\O_{v_{t^{+}}}\) is not
trivially valued in the above proposition. If \((K,v)\) had rank \(1\) then
$v_{t^+}$ for any pseudo-uniformizer \(t\) would satisfy all other requirements
of \Cref{prop:tameable-ramif}. However, not all rank \(1\) valued fields are
deeply ramified.

Fortunately, if the coarsening is trivially valued there are still
straightforward ways to determine whether the field is deeply ramified.

\begin{lem}
    Let \((K,v)\) be a valued field. Let \(t\) be a pseudo-uniformizer such that
    \(vt\) is not minimal positive, \(v(t^{n})\) is cofinal in \(vK\) and
    \(\O_{v}/(t)\) is semi-perfect. Then \((K,v)\) is deeply ramified.

    \begin{proof}
        Because \(\O_{v}/(t)\) is semi-perfect, 
        conclude that \(\O_{v}/(t^{n})\) is semi-perfect for all \(n>0\):
        note that the coarsest coarsening $w_1$ of $v$ with $w_1(t)>0$ and 
        the coarsest coarsening $w_n$ of $v$ with $w_n(t)>0$ coincide, so
        $Kw_1 = Kw_n$. Since \cite[Lemma 7.2.19]{JK23} shows that \(\O_{v}/(t^{n})\) is semi-perfect if and only if $Kw_n$ is perfect, the semiperfectness of \(\O_{v}/(t)\)
        implies that \(\O_{v}/(t^{n})\) is semiperfect.
        As
        \(v(t^{n})\) is cofinal, for any \(x\) with \(0<vt<\infty\) there is an
        \(n>0\) such that \(v(t^{n})\geq vx\), and we get a surjection
        \(\O_{v}/(t^{n})\twoheadrightarrow\O_{v}/(x)\). This means that
        \(\O_{v}/(x)\) is also semi-perfect.
    \end{proof}
\end{lem}

\section{AKE Principles for deeply ramified fields and applications}

We now prove AKE principles for the class \(\C_{p,e}\) of pointed almost separably tame deeply ramified henselian valued fields with fixed degree of
imperfection. We start with relative subcompleteness, which serves as the
starting point for all our proofs later on:

\begin{thm} \label{thm:ake-sub}
    Let \((K,v,t)\subseteq(L,w,t),(L',w',t)\) be separable extensions of pointed
    valued fields in \(\C_{p,e}\) such that \(vK\preceq_{\exists}wL\). If $e$ is
    infinite, assume further that $L$ and $L'$ have the same relative degree of
    imperfection over $K$, i.e., $[L:L^p.K] = [L':(L')^p.K]$. Then, the following are
    equivalent
    \begin{enumerate}[(i)]
        \item \((L,w,t)\equiv_{(K,v,t)}(L',w',t)\);
        \item \(\O_{w}/(t)\equiv_{\O_{v}/(t)}\O_{w'}/(t)\) and
              \(wL\equiv_{vK}w'L'\). \end{enumerate} \label{thm:ake} Moreover,
    if $vK$ is regularly dense, the condition
    \(vK\preceq_{\exists}wL\) can be omitted. If both $wL$ and $w'L'$
    are regularly dense, the value group condition in (ii) can be
    omitted.

    \begin{proof}
        (i) $\implies$ (ii): Clear, as the residue rings $\O_{w}/(t)$ and
        $\O_{w'}/(t)$ and the value groups $wL$ and $w'L'$ are (uniformly)
        interpretable in $(L,v,t)$ resp.~$(L',v',t)$.

        (ii) $\implies$ (i): The proof follows along the same lines as
        \cite[Theorem 5.1.2]{JK23}, we give details where it deviates. By
        \Cref{cor:elem} and the Keisler-Shelah Theorem, we may replace
        \((K,v,t),(L,w,t),\textrm{and }(L',w',t)\) by non-principal ultrapowers
        (with respect to the same ultrafilter) such that we get isomorphisms of
        rings \(f_r: \O_{w}/(t)\rightarrow_{\O_{v}/(t)}\O_{w'}/(t)\) and of
        value groups \(f_g: wL\rightarrow_{vK}w'L'\). Note that, even in case $e
            =\infty$, we still have $[L:K^p] = [L':K^p]$ after passing to an
        ultrapower.

        We write $v_{t^-}$ (resp.~$w_{t^-}$ or $w'_{t^-}$) for the coarsest
        coarsening of $v$ (resp.~$w$ and $w'$) with $v_{t^-}(t)>0$, and we write
        $\overline{v}$ (resp.~$\overline{w}_{t^-}$ or $\overline{w}'_{t^-}$) for
        the valuation induced by $v$ on $Kv_{t^-}$ (resp. by $w$ on $Lw_{t^-}$
        or by $w'$ on $L'w'_{t^-}$). Exactly like in the proof of \cite[Theorem
            5.1.2]{JK23}, we have $Kv_{t^-} = \textrm{Frac}((\O_v/(t))_\mathrm{red})$
        (and similarly for $Lw_{t^-}$ and $L'w'_{t^-}$), so the semiperfectness
        of the residue rings \(\O_{v}/(t),\ \O_{w}/(t)\textrm{ and
        }\O_{w'}/(t)\) implies that  $Kv_{t^-}$, $Lw_{t^-}$ and $L'w'_{t^-}$ are
        perfect. Also, just like in \cite[Theorem 5.1.2]{JK23}, the isomorphisms
        $f_r$ and $f_g$ induce isomorphisms of valued fields $(Lw_{t^-},
            \overline{w}) \cong_{(Kv_{t^-},\overline{v})} (L'w'_{t^-},
            \overline{w}')$ and of ordered abelian groups $w_{t^-}L \cong_{v_{t^-}K}
            w'_{t^-}L'$ and $w_{t^-}L/v_{t^-}K$ is torsion-free.

        We now consider the finest coarsenings $v_{t^+}$ (resp.~$w_{t^+}$ and
        $w'_{t^+}$) for the coarsest coarsening of $v$ (resp.~$w$ and $w'$) with
        $v_{t^+}(t)=0$, whose valuation rings are obtained by adjoining
        $\frac{1}{t}$ to the valuation rings of $v$, $w$ and $w'$. As each of
        $(K,v,t), (L,w,t)$ and $(L',w',t)$ is in $\C_{p,e}$, the valuations $v_{t^+}$, $w_{t^+}$ and
        $w'_{t^+}$ are all separably tame. Moreover, the rank $1$ valuations
        induced by $v_{t^-}$ on $Kv_{t^+}$ (resp.~by $w_{t^-}$ on $Lw_{t^+}$ or
        by $w'_{t^-}$ on $L'w'_{t^+}$) have $p$-divisible value groups (as each
        of $vK$, $wL$ and $w'L'$ is p-divisible), and are defectless by
        saturation, each with perfect residue field as argued above. Thus, the
        valuations $v_{t^-}$, $w_{t^-}$ and $w'_{t^-}$ are all three also
        separably tame.

        Once again like in the proof of \cite[Theorem 5.1.2]{JK23}, we obtain
        $$(L,w_{t-}, \overline{w}) \equiv_{(K,v_{t-}, \overline{v})}
            (L',w'_{t-}, \overline{w}').$$
        By \cite[Theorem 4.22]{A25}, this implies
        $$(L,w,t)\equiv_{(K,v,t)} (L',w',t).$$ The ``moreover'' part follows
        exactly as in the proof of \cite[Theorem 5.1.2]{JK23}.
    \end{proof}
\end{thm}

\subsection*{Elementary substructures}
As an immediate consequence of \Cref{thm:ake-sub}, we obtain the following:
\begin{kor}
    \label{thm:ake-ext}
    If \((K,v,t)\subseteq(L,w,t)\) are both valued fields in \(\C_{p,e}\) such
    that \(L/K\) is separable. Then the following are equivalent:
    \begin{enumerate}
        \item \((K,v)\preceq(L,w)\),
        \item \(vK\preceq vL\) and \(\O_{v}/(t)\preceq\O_{w}/(t)\).
    \end{enumerate}
    Moreover, if both $vK$ and $wL$ are regularly dense, the value group
    condition in (ii) can be omitted.

    \begin{proof}
        Take \((K',v')=(K,v)\). Then the statement is exactly \Cref{thm:ake}.
    \end{proof}
\end{kor}
As mentioned in the introduction, the corollary above now allows us to draw
conclusions about embeddings of certain rank $1$ deeply ramified fields. 

\begin{kor}
    \label{cor:omg}
    Let \(K/\F_{p}(t)^{h}\) be an algebraic extension of valued fields with
    ramification such that the valuation ring \(\O_{v}\) of \(K\) is
    semi-perfect mod \(t\). Then we get an elementary embedding of valued fields
    \[
        K\preceq K.\F_{p}\powerfield{t}.
    \]
    \begin{proof} We claim the following:
        The inclusion of valued fields induces an isomorphism of value group and
        valuation rings mod \(t\). Furthermore, the extension is separated. Next
        we show both fields are deeply ramified and \(\O_{v_{t^{+}}}\) is trivially
        valued for both fields. Both fields are therefore contained in the same
        class \(\C_{p,e}\) and all the requirements of \Cref{thm:ake-ext} are met,
        and the embedding is elementary.

        Now we show all these claims.

        We know that \(\F_{p}(t)^{h}\subseteq\F_{p}\powerfield{t}\) is an
        immediate extension. Therefore, the extension of value groups of
        \(K.\F_{p}\powerfield{t}/K\) is trivial as well.

        Next we show that the ring extension \(\O_{K.\F_{p}\powerfield{t}}/(t)/\O_{K}/(t)\) is trivial.

        We demonstrate that \(\O_{K}/(t)=\O_{K.\F_{p}\powerfield{t}}/(t)\) first for finite
        simple extensions \(K/\F_{p}(t)^{h}\) and show we can repeat this
        process for arbitrary algebraic extension.

        Let \(\F_{p}(t)^{h}(a)\) be a simple algebraic extension. We claim that
        for all \(f\in \F_{p}\powerfield{t}(a)\), there is an
        \(\widetilde{f}\in\F_{p}(t)^{h}(a)\) such that \(v(f-f')\geq t\). The
        element \(f\) is a finite linear combination of powers of \(a\) and
        elements of \(\F_{p}\powerfield{t}\): \(f=\sum_{i=0}^{n}x_{i}f^{i}\) for
        \(x_{i}\in\F_{p}\powerfield{t}\). Polynomials are continuous, and
        therefore there is some \(\delta\) such that if
        \(v(\widetilde{x_{i}}-x_{i})\geq\delta\) and
        \(v(\widetilde{f}-f)\geq\delta\) then
        \[
            v\left(\sum_{i}x_{i}f^{i}-\sum_{i}\widetilde{x_{i}}f^{i}\right)\geq v(t).
        \]
        Because \(\F_{p}(t)^{h}\) is dense in \(\F_{p}\powerfield{t}\), we may
        take all \(\widetilde{x_{i}}\) in \(\F_{p}(t)^{h}\). This gives that
        \(\O_{K}/(t)\to\O_{K.\F_{p}\powerfield{t}}/(t)\) is a surjective map. It is
        injective by construction, so it is an isomorphism as well.

        More generally: if we want to show that some residue
        class of \(\O_{K.\F_{p}\powerfield{t}}/(t)\) is in the image of the
        inclusion, it suffices to consider a finite simple extension generated
        by an element of this residue class.

        Now we show is that the extension is separated.  We make a case distinction depending on whether 
        \(K\) is perfect. 
        If \(K\) is perfect, then \(K.\F_{p}\powerfield{t}/K\) is separable by the definition of perfect fields.
        If \(K\) is non-perfect, then \(\sqrt[p^{n}]{t}\) is a \(p\)-basis for some
        \(n\geq0\) and also in \(K.\F_{p}\powerfield{t}\).
        In either case, the extension \(K.\F_{p}\powerfield{t}/K\) is
        separated.

        By assumption \(v(t)\) is not minimal positive in \(vK\). Thus by
        \Cref{prop:tameable-ramif} \((K,v)\) is deeply ramified.

        The coarsening \(\O_{v_{t^{+}}}\) is trivially valued because \(v\)
        is a rank \(1\) valuation.

        In conclusion, we have shown all necessary requirements to apply
        \Cref{thm:ake-ext}, and hence the embedding is elementary.
    \end{proof}
\end{kor}

\begin{kor}
    \label{cor:as-closure-nice}
    The natural embedding
    \(\asclose{(\F_{p}(t)^{h})}\subseteq\asclose{(\F_{p}(\powerfield{t}))}\) is
    elementary.

    \begin{proof}
        We show that
        \(\asclose{(\F_{p}(t)^{h})}.\F_{p}\powerfield{t}=\asclose{\F_{p}\powerfield{t}}\)
        and use \Cref{cor:omg}. To show the equality, we use Krasner's lemma
        \cite[Theorem 4.1.7]{EP05}.

        Clearly the former is contained in the latter. We show the first has all
        roots of Artin-Schreier polynomials. By the same argument as in
        \Cref{cor:omg}, the inclusion
        \(\asclose{(\F_{p}(t)^{h})}\subseteq\asclose{(\F_{p}(t)^{h})}.\F_{p}(t)^{h}\)
        has dense image. Now take any
        \(f\in\asclose{(\F_{p}(t)^{h})}.\F_{p}(t)^{h}\). By Krasner's lemma it
        is sufficient to show that \(X^{p}-X-\tilde{f}\) has a root for
        \(\tilde{f}\) sufficiently close to \(f\). We can get such an
        \(\tilde{f}\) from \(\asclose{(\F_{p}(t)^{h})}\).
    \end{proof}
\end{kor}

\Cref{thm:ake-ext} also applies naturally in valued fields of higher rank:
\begin{kor}
     \label{cor:as-closure-nicer}
    The natural embedding
    \(\asclose{(\F_{p}\powerfield{t})}\subseteq\asclose{(\F_{p}(\powerfield{t}))}\powerfield{\mathbb{Q}}\) is
    elementary.
\begin{proof}
    Both $(\asclose{(\F_{p}\powerfield{t})}, v_t, t)$ and $\asclose{(\F_{p}(\powerfield{t}))}\powerfield{\mathbb{Q}}, v_t, t)$ are in $\mathcal{C}_{p,1}$.
    The induced embedding between their residue rings modulo $t$ is the identity, 
    and the induced embedding of value groups $\mathbb{Z} \subseteq \mathbb{Q} \oplus_\textrm{lex} \mathbb{Z}$ is elementary. Thus, we conclude by \Cref{thm:ake-ext}.
\end{proof}
\end{kor}

\subsection*{Elementary equivalence}
As separably tame valued fields of positive characteristic even satisfy an
$\textbf{AKE}_\equiv$ principle, we also obtain a relative completeness result
for fields in $\C_{p,e}$:

\begin{thm} Fix a prime $p$ and $e \in \mathbb{N} \cup \{\infty\}$, and let
    \((K,v,t),(L,w,s)\) be pointed valued fields in \(\C_{p,e}\). If
    \((vK,v(t))\equiv(wL,w(s))\) and \(\O_{v}/(t)\equiv\O_{w}/(s)\) then
    \((K,v)\equiv(L,w)\). \label{thm:ake-equiv}

    \begin{proof}
        The proof is analogous to that of \Cref{thm:ake} where we take the
        common substructure to be \((\F_{p},v_{\text{triv}})\) with the trivial
        valuation. Each step is similar to the corresponding one in the proof of
        \Cref{thm:ake}. At the end, we get an isomorphism of valued fields over
        \(\F_{p}\).

        The original intermediate isomorphism
        \((\O_{v'}/(t))_{\text{red}}\cong_{(\O_{v}/(t))\text{red}}(\O_{w}/(t))_{\text{red}}\)
        now becomes an isomorphism over \(\F_{p}\) instead of
        \((\O_{v}/(t))_{\text{red}}\). We can just as well apply
        \cite[Theorem 4.22]{A25} to this as well. Note that this works for
        \(\F_{p}\) in particular because it is perfect and trivially valued
        (hence tame), and therefore any extension is separable and
        \(vK/(v_{\text{triv}}\F_{p})=vK\) is torsion free.
    \end{proof}
\end{thm}

Note that in the special case of \Cref{thm:ake-equiv} in case that $K$ and $L$
are both perfect (i.e., \cite[Theorem 5.1.7]{JK23}), the conclusion can be
strengthened to \((K,v,t)\equiv(L,w,s)\) (and hence it is possible to rephrase
\cite[Theorem 5.1.7]{JK23} as an equivalence): in case $e=0$, the field
$(\mathbb{F}_p(t)^\textrm{h})^\mathrm{perf} \cong
    (\mathbb{F}_p(s)^\textrm{h})^\mathrm{perf}$ is a common subfield of both
\((K,v,t)\) and \((L,w,s)\) which is also in $\C_{p,0}$, and so the elementary
equivalence can be strengthened to include a symbol for the constants \(t\) or
\(s\) respectively. Whether this works in greater generality remains open:

\begin{Qu}
Can \Cref{thm:ake-equiv} be strengthened to an equivalence in general? More precisely: 
Given $p$ prime
and $e \in \mathbb{N} \cup \{\infty\}$, $e>0$, and two pointed valued fields
$(K,v,t)$ and $(L,w,t)$ in $\mathcal{C}_{p,e}$. Does the equivalence
$$((vK,v(t))\equiv (wL,w(s)) \textrm{ and } \O_v/(t) \equiv \O_w/(s)) \Longleftrightarrow (K,v,t) \equiv (L,w,s)$$ hold?
\end{Qu}
    
\subsection*{Existential closedness}
Last but not least, we obtain an AKE principle transferring existential
closedness.
\begin{thm}
    Let \((K,v)\subseteq(L,w)\) be a separable inclusion of deeply ramified
    valued fields in a fixed class \(\C_{p,e}\). Then the following statements
    are equivalent:
    \begin{enumerate}
        \item \((K,v)\preceq_{\exists}(L,w)\),
        \item \(vK\preceq_{\exists}wL\) and
              \(\O_{v}/(t)\preceq_{\exists}\O_{w}/(t)\).
    \end{enumerate}

    \begin{proof} This proof is similar to that of
        \Cref{thm:ake}. The implication (1) $\implies$ (2) is clear. We now prove (2)
        $\implies$ (1). By \Cref{cor:elem}, we may assume that \(L\) is saturated of size
        \(|{K}^{+}|\).

        Then can find a sufficiently saturated ultrapower \((K^{*},v^{*})\) of
        \((K,v)\) such that we have a (non-elementary) embedding \(\O_{w}/(t)\to
        \O_{v^{*}}/(t)\) and \(wL\to v^{*}K^{*}\) over \((K,v)\). Then the reduced
        ring \((\O_{w}/(t))_{\text{red}}\) embeds into
        \((\O_{v^{*}}/(t))_{\text{red}}\) over \((\O_{v}/(t))_{\text{red}}\) as
        well. This is exactly an embedding of valued fields
        \((Lw_{t^{-}},\overline{w})\to(K^{*}v^{*}_{t^{-}},\overline{v^{*}})\) over
        \((K,v)\).
        \[
            \begin{tikzcd}
                (Kv_{t^{-}},\overline{v}) & (Lw_{t^{-}},\overline{w})         \\
                & (K^{*}v^{*}_{t^{-}},\overline{v^{*}})
                \arrow[from=1-1,to=1-2]
                \arrow[from=1-1,to=2-2]
                \arrow[from=1-2,to=2-2]
            \end{tikzcd}
        \]
        As before, the valued field \((Lv_{t^{+}},\overline{v_{t^{-}}})\) is tame by
        saturation so the composition of valuations \((L,v_{t^{-}})\) is separably
        tame.

        Now, we can use the separable relative embedding property \cite[Theorem 4.19]{A25} to embed
        \((L,w_{t^{-}})\) into \((K^{*},v^{*}_{t^{-}})\) over \((K,v_{t^{-}})\). To see
        this note that by assumption \(L/K\) is separable and the residue fields of
        all non-trivial coarsenings are perfect because the valued fields are
        deeply ramified. The value group extension is torsion free because we
        assumed that \(vK\preceq_{\exists}vL\).

        We then obtain a commutative diagram using the separable relative
        embedding property
        \[
            \begin{tikzcd}
                (K,v_{t^{-}}) & (L,w_{t^{-}})         \\
                & (K^{*},v^{*}_{t^{-}})
                \arrow[from=1-1,to=1-2]
                \arrow[from=1-1,to=2-2]
                \arrow[from=1-2,to=2-2]
            \end{tikzcd}
        \]
        The map \(K\to K^{*}\) is induced by the map \(O_{w}/(t)\to\O_{v^{*}}/(t)\),
        hence we get that this is an embedding with respect to the total
        valuations \(v,w,v^{*}\). Now because \((K,v)\to(K^{*},v^{*})\) is
        elementary by construction, we can conclude \((K,v)\preceq_{\exists}(L,w)\).
    \end{proof}
\end{thm}

	\printbibliography
	
	
	
	

    \todos
\end{document}